\documentclass[12pt]{amsart}

\usepackage{amssymb, amsmath, amsthm,stmaryrd}
\usepackage[backref]{hyperref}
\usepackage[alphabetic,backrefs,lite]{amsrefs}
\usepackage[utf8]{inputenc}

\usepackage{amscd}   % for commutative diagrams
\usepackage{fullpage}
\usepackage[all]{xy} % for complicated commutative diagrams
\usepackage{MnSymbol,graphicx}
\usepackage{mathrsfs}

\usepackage{color}

\usepackage{makecell}

\usepackage{tikz}
\usepackage{verbatim}

  \tikzstyle{block} = [rectangle, draw,
      text width=8em, text centered, minimum height=4em]
  \tikzstyle{line} = [draw, -latex']

\usepackage{amssymb, amsmath, amsthm,stmaryrd}
\usepackage[backref]{hyperref}
\usepackage[alphabetic,backrefs,lite]{amsrefs}
\usepackage{amscd}   % for commutative diagrams
\usepackage{fullpage}
\usepackage[all]{xy} % for complicated commutative diagrams
\usepackage{MnSymbol,graphicx}
\usepackage{mathrsfs,enumitem}
\usepackage{multirow,qtree}% http://ctan.org/pkg/multirow
\DeclareFontEncoding{OT2}{}{} % to enable usage of cyrillic fonts

\usepackage{tikz}

\newtheorem{lemma}{Lemma}[section]
\newtheorem{theorem}[lemma]{Theorem}
\newtheorem{proposition}[lemma]{Proposition}
\newtheorem{prop}[lemma]{Proposition}
\newtheorem{cor}[lemma]{Corollary}
\newtheorem{conj}[lemma]{Conjecture}

\newtheorem{claim*}{Claim}

\newtheorem{defn}[lemma]{Definition}

\theoremstyle{definition}
\newtheorem{remark}[lemma]{Remark}

\newtheorem{rmk}[lemma]{Remark}

\newcommand{\cB}{\mathcal{B}}

\newcommand{\Q}{{\mathbb Q}}

\newcommand{\Z}{{\mathbb Z}}

\newcommand{\fp}{\mathfrak{p}}
\newcommand{\frakf}{{\mathfrak f}}

\newcommand{\frapk}{{\mathfrak p}}
\newcommand{\frakq}{{\mathfrak q}}

\DeclareMathOperator{\Gal}{Gal}

\DeclareMathOperator{\GL}{GL}

\numberwithin{equation}{section}
\numberwithin{table}{section}

\title{Non-trivial Integer Solutions of  $x^r+y^r=Dz^p$}

\author{Yasemin Kara}
\author{Diana Mocanu}
\author{Ekin \"Ozman}

\address{Bogazici University\\
Department of Mathematics\\
Bebek, 34342 \\
Instanbul, Turkey.}

\email{yasemin.kara@bogazici.edu.tr}

\address{Max Planck Institute for Mathematics\\
Vivatsgasse 7, 53111\\
Bonn, Germany.}

\email{diana.mocanu97@outlook.com}

\address{Bernoulli Institute \\
Nijenborgh 9, 9747 AG \\
Groningen, the Netherlands.}

\email{e.ozman@rug.nl}


\begin{document}

	\maketitle

	\begin{abstract}
In this paper, we use the modular method over totally real fields together with some standard conjectures (the Weak Frey--Mazur Conjecture and the Eichler--Shimura Conjecture) to prove that infinitely many equations of the type $x^r+y^r=Dz^p$ do not have any non-trivial primitive integer solutions, where $r \geq 5$ is a fixed prime, whenever $p$ is large enough. For $r \equiv 3 \pmod 4$, we get the same result with only assuming the Weak Frey--Mazur Conjecture. 
	\end{abstract}

\section{Introduction}

Let \( r \geq 5 \) be a fixed rational prime, \( p \) a rational prime different from \( r \), and \( \gcd(r, D) = 1 \). Assume \( p \) is large enough that \( D \) contains no \( p \)th power. The equation 
\begin{equation}
    \label{maineqn}
x^r + y^r = Dz^p, \quad x, y, z \in \mathbb{Z},
\end{equation}
is referred to as the \emph{generalized Fermat equation with signature \( (r, r, p) \) and coefficient \( D \)}. A solution \( (x, y, z) \) is called \emph{primitive} if \( \gcd(x, y, z) = 1 \), and \emph{trivial} if \( xyz = 0 \) or \( z = \pm 1 \). Note that we extended the standard definition of trivial solution in the literature to exclude situations of the form $x=1$, $D=\pm(1+y^r)$ and $z=\pm 1$. %Let  $r \geq 5 $ be a fixed rational prime number, $p$ be a rational prime number different than $r$, and $\gcd(r,D)=1$. We assume that $p$ is large enough such that there is no integer $a$ with $a^p \mid D$ i.e. $D$ does not contain any $p$th power.
%We refer the equation \begin{equation}\label{maineqn} x^r+y^r=Dz^p, \quad x,y,z\in \Z \end{equation}
%as \textit{the Fermat equation over $\Q$ with signature $(r,r,p)$ and coefficient $D$}.  Let $(x,y,z)$ be a solution to $x^r+y^r=Dz^p$  . We say that
%$(x,y,z)$ is a \textit{primitive} solution if $(x,y,z)=1$ and a \textit{trivial} solution if $xyz =0 $ or $z=\pm 1$. 

A summary of the instances of \eqref{maineqn} that have been solved can be found in Table \ref{my-table2}.  The goal of this article is two-fold. On one hand, we aim to complete this resolution, at the expense of assuming some standard (nevertheless strong) conjectures in the field. On the other hand, we intend to offer a short survey of this variant of the modular method over totally real fields, specifically focusing on the use of the Eichler--Shimura correspondence in conjunction with the Frey--Mazur conjecture. While this approach is familiar to experts, it seems that it has not been extensively discussed in the existing literature.

The generalized Fermat equation of signature $(r,r,p)$ has been fruitfully studied in recent years using various generalizations of the so-called \textit{modular method}. This is the approach pioneered by Andrew Wiles in the famous proof of Fermat's Last Theorem \cite{AW95}, building on the work of Frey, Mazur, Ribet, Serre, and many others.
The modular method for attacking Diophantine equations (usually parametrized by a varying prime exponent $p$) can be summarized in three steps.
\begin{itemize}
	\item[\textbf{Step 1.}] \textbf{Constructing a Frey elliptic curve.} Attach to a putative solution (of some Diophantine equation) an elliptic curve $E/K$, for $K$ an appropriately chosen totally real number field. We require $E$ to have the minimal discriminant $\Delta=CB^p$, where $B,C \in K$ and $C$ is independent of the solution. More formally, one requires that the residual Galois representations $\bar{\rho}_{E,p}$ has bounded ramification which is independent of the putative solution. For example, given a solution of Fermat's equation $x^p+y^p=z^p$, $K=\Q$ and the Frey elliptic curve is given by $E: Y^2=X(X-x^p)(X+y^p)$. 
     
\item[\textbf{Step 2.}]\textbf{Modularity/Level lowering.} Prove modularity of $E/K$ and irreducibility of the residual Galois representations $\bar{\rho}_{E,p}$ attached to $E$, to conclude (via level lowering results), that $\bar{\rho}_{E,p} \sim \bar{\rho}_{\mathfrak{f},p}$ where $\mathfrak{f}$ is a Hilbert newform with rational Hecke eigenvalues and level $\mathcal{N}_p$, depending only on $C$. For example, in the setting of Fermat's Last Theorem, Wiles' Modularity Theorem together with Ribet's Level Lowering Theorem give rise to a classical newform $f$ of weight $2$ and level $N_p=2$.
 \item[\textbf{Step 3.}]\textbf{Elimination.} Prove that among the finitely many Hilbert newforms  $\mathfrak{f}$ predicted above, none of them satisfies $\bar{\rho}_{E,p} \sim \bar{\rho}_{\mathfrak{f},p}$. Continuing our example, we get a contradiction to the existence of solutions to Fermat's equation by simply noting that there are no newforms of weight $2$ and level $2$.
\end{itemize}
\begin{table}[h]
\resizebox{\textwidth}{!}{%
\begin{tabular}{|l|l|l|l|}
\hline
$r$                                                                                                                           & $D$          & References                                                     \\ \hline                     
3 & 1 & Freitas \cite{Freitas33p}, Kraus* \cite{Kraus33p}                                       \\ \hline
5                                                                                                                             & 1,2          & \begin{tabular}[c]{@{}l@{}}Billerey, Chen, Dembele, Dieulfait,  Freitas* \cite{BCDDF22}\end{tabular}  \\ \hline
5,13                                                                                                                          & 3            & \begin{tabular}[c]{@{}l@{}}Billerey, Chen, Dieulfait, Freitas \cite{BCDF19}\end{tabular}                                                             \\ \hline
7                                                                                                                             & 3            & \begin{tabular}[c]{@{}l@{}}Billerey, Chen, Dieulfait, Freitas, Najman \cite{BCDFN}, Freitas \cite{F}\end{tabular}                                                                                                                 \\ \hline
11                                                                                                                            & 1            & \begin{tabular}[c]{@{}l@{}}Billerey, Chen, Dieulfait, Freitas, Najman* \cite{BCDFN}\end{tabular}           \\ \hline
\begin{tabular}[c]{@{}l@{}} $r \in \{7,13,19,23,37,47,59, $\\ 61,67,71, 79,83,101,103,107,\\ 131,139,149$\}$ \end{tabular} & $1$            & Mocanu* \cite{rrp}                                                                                                               \\ \hline
$r \nmid D$                                                                                                                   & odd, $\neq1$ & Freitas, Najman* \cite{FNaj}  \\ \hline                                                                                              
\end{tabular}%
}
\vspace{0.5cm}  % Add vertical space to move the caption lower
\caption{\small Unconditionally solved instances of $x^r+y^r=Dz^p$ for $p$ large enough. The upper script * stands for results subject to local conditions.}
\label{my-table2}
\end{table}
In general, in the last step, one has to compute all Hilbert newforms at a given level, which is a computationally difficult task. Various approaches in literature consist in creative ways of tackling Step 3 without computing Hilbert newforms. In some situations, a converse of modularity is known to hold, namely an Eichler--Shimura correspondence, which relates rational Hilbert newforms back to elliptic curves. In these situations, one attaches a second elliptic curve $E'$ which has conductor independent of the solution (more precisely, it equals the quantity $\mathcal{N}_p$ described in Step 2). In this way, one converts the problem of eliminating Hilbert newforms into a problem about eliminating elliptic curves with special behaviour at primes of bad reduction. There are various ways in the literature to tackle the elimination of such elliptic curves.
For example, in \cite{FNaj} the authors use the study of symplectic types of elliptic curves, while in \cite{rrp}, the author uses the theory of $S$-unit equations coming from  parametrizations of elliptic curves with specific behaviour at bad primes. In this article, we will show that whenever $p$ is large enough, one can eliminate elliptic curves using the Weak Frey--Mazur Conjecture (Conjecture \ref{weakFMconj}).

Another popular direction of attacking this problem originated in Darmon's program and involves using Frey abelian varieties of higher dimensions (in place of Frey elliptic curves). Recent work of Billerey, Chen, Demb\'el\'e, Dieulefait, Freitas and Najman \cite{BCDFN}, \cite{BCDF19},\cite{BCDDF22} use recent advancements on the modularity of Galois representations coming from Frey abelian varieties of $\GL_2$-type to get asymptotic resolutions for equation \eqref{maineqn}. 
Recently, the authors in \cite {FNaj} proved that for infinitely many integers $D$, the above equation has no non-trivial primitive solutions such that $2\mid x + y $ or $r\mid x + y$, for a set of exponents $p$ of positive density. 

\subsection{Our Results}
We employ the modular method as outlined in the previous section. In Step 1, we use Freitas' recipes \cite{F} to construct Frey elliptic curves over certain totally real fields depending on 
$r$. The second step relies on modularity and level lowering results for elliptic curves over totally real fields from \cite{FSasymp}. In Step 3, we apply various results on the Eichler--Shimura correspondence and image of inertia comparison arguments in order to transform the problem into one about eliminating elliptic curves. Then, 
we apply a weak form of the Frey--Mazur conjecture (Conjecture \ref{weakFMconj}, \cite{MR1638494FreyMazur}) to establish the following theorem.
%We will employ the modular method as described in the previous section. For Step 1, we will use the recipes given by Freitas in \cite{F} to construct appropriate Frey curve over certain totally real fields (which depend on $r$). For Step 2 we use theoretical results about the modularity of elliptic curves over totally real fields described in \cite{FSasymp}. Lastly, for Step 3 we use aweak version of the Frey--Mazur conjecture (Conjecture \ref{weakFMconj}, \cite{MR1638494FreyMazur}), Eichler--Shimura type theorems and image of inertia comparison arguments, to get the following theorem. 
\begin{theorem}[First Main Theorem]\label{thm:main1}
    Let $r \geq 5$ be a fixed rational prime and $D$ be a nonzero integer with $\gcd(D,r)=1$. Assume the Weak Frey--Mazur Conjecture (Conjecture \ref{weakFMconj}). Then,
    \begin{enumerate}
    \item If $r\equiv 3 \mod 4$, then there exists a (non-effective) constant depending on $r$ and $D$, $\cB_{r,D}$, such that for any rational prime $p$ with $p>\cB_{r,D}$, the equation
    $x^r+y^r=Dz^p$
    has no non-trivial, primitive, integer solutions $(x,y,z)$.
    %\diana{I am a bit confused about the fact that solutions of the form $x=1, y=2, z=1, D=1+2^r$ always exist (as long as $\gcd(D,r)=1$). I think we have to exclude $z=1$ somehow? I think all "bounding p" arguments subtly rely on this assumption.}
    \item If $r\equiv 1 \mod 4$, then there exists a (non-effective) constant depending on $r$ and $D$, $\cB'_{r,D}$, such that for any rational prime $p$ with $p>\cB'_{r,D}$, the equation
    $x^r+y^r=Dz^p$
    has no non-trivial, primitive, integer solutions $(x,y,z)$ where $r\mid z$. 
    \end{enumerate}
    
\end{theorem}
\begin{remark}
    We divide the proof into two parts, corresponding to the instances in which the Eichler--Shimura conjecture (Conjecture \ref{ESconj}) is known to hold, as described in
    Therorem \ref{thm:EC}. More details are given in Remark \ref{rmk:r}.
\end{remark}
Assuming the Eichler–-Shimura Conjecture (Conjecture \ref{ESconj}) one gets the following.
\begin{theorem}[Second Main Theorem]\label{thm:main2}
Let $r \geq 5$ be a fixed rational prime such that $r\equiv 1 \mod 4$ and $D$ be a nonzero integer with $\gcd(D,r)=1$. Assume the Weak Frey--Mazur Conjecture (Conjecture \ref{weakFMconj}) and the Eichler–Shimura Conjecture (Conjecture \ref{ESconj}). Then there exists a (non-effective) constant depending on $r$ and $D$, $\cB''_{r,D}$ such that for any rational prime $p$ with $p>\cB''_{r,D}$, the equation
    $$x^r+y^r=Dz^p$$
    has no non-trivial, primitive, integer solutions $(x,y,z)$. 
    
\end{theorem}
\begin{remark}
    The constants $\mathcal{B}_{r,D}$, $\mathcal{B}_{r,D}'$, $\mathcal{B}_{r,D}''$ are ineffective as they depend on the constant in the Weak Frey--Mazur Conjecture (Conjecture \ref{weakFMconj}). 
\end{remark}
\subsection{Acknowledgements.} The authors are sincerely grateful to Samir Siksek for useful discussions and suggestions. The first and last authors are partially supported by TUBITAK Project Number 122F413.  The first author is also supported by Bogazici University Research Fund Grant Number 19082. The second author was involved in writing this article while transitioning from University of Warwick to Max Planck Institute for Mathematics in Bonn, and thus she is thankful to both institutions for their kind support.
\newpage

\section{Frey Curve and Related Facts}\label{frey}

This section follows standard notations and ideas which were constructed in full generality in \cite{F}.  We included the details here for the convenience of the reader. 
Let  $r \geq 5$ be a fixed rational prime. We denote by $\zeta_r$  a primitive $r$th root of unity, so $\mathbb Q(\zeta_r)$ is the $r$th cyclotomic field which has degree $r-1$ over $\Q$. Let $K$ be the maximal totally real subfield of $\mathbb Q(\zeta_r)$, namely $K:=\Q(\zeta_r+\zeta_r^{-1})$. Note that $K$ is an abelian extension of $\Q$ of degree $\frac{r-1}{2}.$ We write $\mathcal O_K$ for the ring of integers of $K$, then $\mathcal O_K=\Z[\zeta_r+\zeta_r^{-1}]$.  The prime $r$ is totally ramified in $K$ and we denote the unique prime lying over $r$ with $\beta$.  Hence
\begin{equation}\label{beta}
    r\mathcal O_K= \beta^{(r-1)/2}, \quad \beta=(1-\zeta_r)(1-\zeta^{-1}_r)\mathcal O_K.
\end{equation}
 Since $\beta$ is the unique ideal above $r$ we can also write $\beta=(1-\zeta_r^{i})(1-\zeta_r^{-i})\mathcal O_K$ for any $0\leq i \leq (r-1)$ via conjugation by elements of $\text{Gal}(K/\Q)$.  
We denote by $v_\beta(\cdot)$ the valuation of $\cdot$ at $\beta$.
%For the rest of this section, we relate the equation $x^r+y^r=Dz^p$ to several homogenous quadratic equations that we will use to construct a Frey curve with the desired properties.

Famous work of Darmon (\cite{Darmon}) shows that there is no rational Frey elliptic curve attached to $x^r+y^r=Dz^p$.
Thus, we would like to replicate the construction of the classical Frey elliptic curve in Wiles' original proof by constructing a similar looking curve over a larger number field.  
Therefore, in the remainder of this section, we connect the equation $x^r+y^r=Dz^p$ to a series of homogeneous quadratic equations.
We start by factorizing the right-hand side of equation $Dz^p=x^r+y^r$  and get that

	$$Dz^p=(x+y)\prod\limits_{1 \leq j \leq  (r-1)/2} (x+\zeta_r^jy)(x+\zeta_r^{-j}y)=(x+y)\prod\limits_{1 \leq j \leq (r-1)/2}(x^2+(\zeta_r^j + \zeta_r^{-j})xy+y^2).$$ For $0 \leq j \leq (r-1)/2$ we denote $\alpha_j:=\zeta_r^j + \zeta_r^{-j}= 2 \cos(\frac{2\pi j}{r}) \in \mathcal O_K$ where $K$ is as above.
Thus, we get that

$$Dz^p=(x+y)\prod \limits_{1 \leq j \leq (r-1)/2}f_j(x,y),$$ where $f_j(x,y):=(x^2+y^2)+\alpha_j xy$, for $0\leq j\leq (r-1)/2$, in particular $f_0(x,y)=(x+y)^2$. For simplicity, from now on we will write $f_i$ in place of $f_i(x,y).$

The following lemmas explain how the factorizations of distinct $f_i$s relate to each other.

\begin{lemma}\label{lem:alpha}
	For any $0 \leq k < j \leq  \frac{r-1}{2}$, the difference $\alpha_k-\alpha_j$ is divisible by $\beta$ only once i.e. $v_\beta(\alpha_k-\alpha_j)=1$.
\end{lemma}

\begin{proof}
	This follows from the fact that $$\alpha_k-\alpha_j=(\zeta_r^k+\zeta_r^{-k})-(\zeta_r^j+\zeta_r^{-j})=\zeta_r^{-j}(1-\zeta_r^{j-k})(\zeta_r^{k+j}-1)$$ and  the description of $\beta$ given in the equation \eqref{beta}.
\end{proof}

\begin{lemma}
    
 \label{prop:relprime}For $0 \leq j \leq (r-1)/2$, the factors $f_j$s are pairwise coprime outside $\beta$. 
\end{lemma}

\begin{proof}Suppose $\pi$ is a prime with $\pi \mid  f_j, f_k$ and $\pi \neq \beta$. Then 
$$\pi \mid (f_j-f_k)= (\alpha_k-\alpha_j)xy=[(\zeta_r^k+\zeta_r^{-k})-(\zeta_r^j+\zeta_r^{-j})]xy=\zeta_r^k(1-\zeta_r^{-k+j})(1-\zeta_r^{-k-j})xy.$$ If $\pi \mid  x$ then since $\pi \mid  f_j$ we get $\pi\mid  y$, which contradicts the assumption that $x$ and $y$ are coprime. Similarly, we get a contradiction if $\pi \mid  y$. Therefore, we can say that $\pi \nmid xy$ hence $\pi \mid  \zeta_r^k(1-\zeta_r^{-k+j})(1-\zeta_r^{-k-j})$ but this is only divisible by $\beta$ and $\pi \neq \beta$, giving a contradiction once again.
\end{proof}

\begin{lemma}\label{lem:b2} 
The valuation $v_{\beta}(f_j)$ is either $0$ or $1$ and $v_{\beta}(f_i)=v_{\beta}(f_j)$ for all $1 \leq i,j \leq \frac{r-1}{2}$.
	\end{lemma}

\begin{proof} Assume that $\beta^2$ divides $f_j$ for some $1\leq j \leq (r-1)/2.$ Since $\beta$ is the unique prime above $r$ we have $\beta^2$ divides $f_j^\sigma$ for every $\sigma \in \Gal(K/\Q).$ This gives $\beta^2 \mid  f_j-f_j^\sigma=(\alpha_j-\alpha_j^\sigma) xy.$ Notice that $\alpha_j^\sigma = \zeta^i +\zeta^{-i}$ where $i \nequiv \pm j \pmod r.$ Therefore, $$\alpha_j-\alpha_j^\sigma=(\zeta_r^j+\zeta_r^{-j})-(\zeta_r^i-\zeta_r^{-i})=\zeta_r^j(1-\zeta_r^{-j+i})(1-\zeta_r^{-j-i}).$$ One computes that $v_{\beta}(\alpha_j-\alpha_j^\sigma)=1$, and thus we must have $\beta \mid  xy$. Combining this with $\beta \mid f_j=x^2+y^2+\alpha_jxy$ we get that $\beta \mid  x $ and $\beta \mid  y$ contradicting the fact that $x$ and $y$ are coprime. Hence, $v_\beta(f_j)<2$. Moreover, since $f_j$ is an algebraic integer, we get that $v_\beta(f_j)\in \{0,1\}$. Suppose $\beta \mid   f_j$ and note that $$\beta \mid  (f_j-f_i)=\zeta_r^k(1-\zeta_r^{-k+j})(1-\zeta_r^{-k-j})xy,$$ for any $i\neq j$. Thus, it should divide $f_i$ for all $i$ and $v_{\beta}(f_i)=v_{\beta}(f_j)=1$.

	\end{proof}

\begin{cor}\label{cor:fjxy}
	For $1\leq j \leq (r-1)/2$, the terms $f_js$ factorize as $$f_j \mathcal O_K= (x^2+y^2+\alpha_jxy)\mathcal O_K= \mathcal I_j^p \mathcal D_j \beta^e, \qquad (x+y)\mathcal O_K = \mathcal I_0^p \mathcal D_0 (r)^{e_0},$$ where $\prod_{k=0}^{(r-1)/2}\mathcal D_k= (D) \mathcal O_K$, $e\in \{0,1\}$, $e_0=pv_r(z)-e$ with $\mathcal I_j$ pairwise coprime ideals which are prime-to-$\beta$.
\end{cor}
\begin{proof}
	Coprimality of $\mathcal I_j$, with $j \neq 0$ follows from Lemmas \ref{prop:relprime} and \ref{lem:b2}. 
	The exponents $e,e_0$ are explained by the fact that $(r)\mathcal O_K = (\beta)^{\frac{r-1}{2}}$ and $\gcd(D,r)=1$, hence $$p\frac{r-1}{2}v_r(z)=pv_\beta(z)=v_\beta(Dz^p)=v_\beta(x^r+y^r)=v_\beta(x+y)+\sum_{j=1}^{(r-1)/2}v_\beta(f_j)$$
	$$=\frac{r-1}{2}v_r(x+y)+\frac{r-1}{2}e = \frac{r-1}{2}e_0+\frac{r-1}{2}e. $$
	
\end{proof}

In the remainder of this section, we introduce a Frey elliptic curve associated with a proposed solution to Equation \ref{maineqn} and establish several results about this curve. We also examine the special properties of this Frey curve when considering specific solutions, such as those where 
$ r \mid  z.$ These type of restrictions occur naturally in the study of Fermat-type equations and analogous to \emph{first case solutions} as called by Darmon in \cite{Darmon}. Studying this special case is essential, as applying Theorem \ref{thm:EC} -- a critical step in proving the second part of the first main theorem -- requires detailed information about the images of inertia at 
$\beta$, specifically $\overline{\rho}_{E,p}(I_\beta)$.  
For instance, when $r \equiv 1 \mod 4$, the application of Theorem \ref{thm:EC} necessitates that the Frey curve exhibits potentially multiplicative reduction, a condition satisfied only under the additional assumption that 
$r \mid  z$.

%In the rest of this section, we introduce a Frey elliptic curve attached to a proposed solution of Equation \ref{maineqn} and prove some results about this curve. We will also consider special properties that this Frey curve has when we start with a special solution such as $r\mid  z.$  We need to study the special solution case too since in order to apply Theorem \ref{thm:EC}, which is a crucial step while proving the second part of the first main theorem, we need information regarding the images of inertia at $\beta$, $\overline{\rho}_{E,p}(I_\beta)$ . For instance, to apply Theorem \ref{thm:EC} when $r \equiv 1 \mod 4$, we need the Frey curve to have potentially multiplicative reduction. This is achieved only if we have the additional assumption $r \mid  z$. \ekin{recall Darmons first case solution}

%In the rest of this section, we introduce two different Frey elliptic curves and prove some results about these curves.  We need two Frey elliptic curves since in order to apply Theorem \ref{thm:EC}, which is a crucial step while proving the second part of the first main theorem, we need information regarding the images of inertia at $\beta$, $\overline{\rho}_{E,p}(I_\beta)$ . For instance, to apply Theorem \ref{thm:EC} when $r \equiv 1 \mod 4$, we need the elliptic curve $E$ to have potentially multiplicative reduction. However, even if we have the extra assumption $r \mid  z$, we get $v_\beta(j_{E_1})=0$.

\subsection{Frey Curve}\label{curve2}
The first step in the modular method is to attach a Frey curve $E$ to $(x,y,z)$, where $(x,y,z)$  is a primitive non-trivial integer solution of $x^r+y^r=Dz^p$. 

We aim to mimic the Frey elliptic curve construction from the original proof of Fermat's Last Theorem. We follow \cite{F} who constructs an elliptic curve
\[
E: Y^2 = X(X-A)(X+B)
\]
defined over the totally real field $K=\Q(\zeta_r+\zeta_r^{-1})$. In this setting, the second step of the modular method is known to work by work of Freitas and Siksek \cite{FSasymp}.

Now, let us define $A,B$ and $C$.
%We will first state the general properties of this curve and then study the special properties that $E$ has when $r\mid  z$. 
Firstly, we consider the following equation between the terms $f_0:=(x+y)^2, f_1$ and $f_2$:
\begin{equation}\label{alphas}   
(\alpha_1-\alpha_2)(x+y)^2 + (\alpha_2-2)f_1+ (2-\alpha_1)f_2 = 0	.
\end{equation}
Here we know $f_1 \mathcal O_K$ and $f_2\mathcal O_K$ have special shapes as described in Corollary \ref{cor:fjxy} and $(\alpha_1-\alpha_2)$, $(\alpha_2-2)$, $(2-\alpha_1)$ are $S$-units where $S=\{\beta\}$. %In fact $v_\beta(\alpha_1-\alpha_2)=1$. 
We define $A,B,C$ as follows:
 \begin{equation}\label{eqn:ABC}A:=(\alpha_1-\alpha_2)f_0,\quad B:=(\alpha_2-2)f_1,\quad C:= (2-\alpha_1)f_2.\end{equation}
In particular, we note that $E$ depends on the putative solution $(x,y,z)$, as the $f_i$s do.
\begin{lemma}\label{lem:unit}
The valuations of the terms appearing in \eqref{alphas} are given as follows: $$v_\beta(\alpha_1-\alpha_2)=v_\beta(\alpha_2-2)=v_\beta(2-\alpha_1)=1.$$
\end{lemma}
\begin{proof}
This follows immediately from the definitions of $\alpha_1,\alpha_2$, and $\beta$.
\end{proof}

%For the rest of this subsection, we assume that $r\mid  z$.

In the following lemma, we determine the valuations of $A,B$ and $C$ under the condition $r\mid  z$. These will be used to show that $E$ has potentially multiplicative reduction at the prime $\beta$ above $r$, which will be used in the Eichler--Shimura step, as detailed in Remark \ref{rmk:r}.

\begin{lemma}\label{lem:vals}  Let $D$ be an integer with $\gcd(D,r)=1$ and suppose that $v_r(z)=k>0$. Then the following statements hold
	\begin{itemize}
		\item $v_\beta(B)=v_\beta(C)=2,$
		\item $k_a:=v_\beta(A)=(pk-1)(r-1)+1$.
		\end{itemize}
	\end{lemma}
\begin{proof}
By Lemma \ref{lem:b2}, $v_\beta(f_j)=0$ or $1$. However when $r\mid  z$ we have $v_\beta(f_j)=1$ for $1 \leq j \leq f_{\frac{r-1}{2}}$. By Lemma \ref{lem:unit}, $v_\beta(B)=v_\beta(C)=2.$ For the second part, we have that $r^{pk} \mid  \mid  Dz^p.$ Since each factor $f_j$ is divisible by $\beta$ exactly once and  $(x+y) f_1\ldots f_{\frac{r-1}{2}}=Dz^p$ we see that $r^{pk-1} \mid  \mid  (x+y)$, which gives $v_\beta(x+y)=(pk-1)\frac{r-1}{2}$.  By Lemma \ref{lem:alpha} we have $v_\beta(\alpha_1-\alpha_2)=1$. This shows that $v_\beta(A)= (pk-1)(r-1)+1$.
\end{proof}

Hence, the terms $A,B$ and $C$ are coprime away from $\beta$ and have the following special factorizations as ideals of $\mathcal O_K:$
$$A\mathcal O_K={\mathcal{D}_0}(\mathcal{I}_0)^p \beta^{k_a}, \quad B\mathcal O_K={\mathcal{D}_1}(\mathcal{I}_1)^p \beta^{2}, \quad C\mathcal O_K={\mathcal{D}_2}(\mathcal{I}_2)^p \beta^{2},$$
where all ideals are as in Corollary \ref{cor:fjxy} and $k_a$ as in Lemma \ref{lem:vals}.

\subsubsection{\textbf{Arithmetic Invariants and behaviour at bad primes}}

We recall here well-known facts about the arithemetic invariants of $E$.
\begin{lemma}  \label{lem:freyinv} Let $E$ be the Frey elliptic curve given by $$E: Y^2=X(X-A)(X+B),$$ where $A,B,C$ are as in Equation \ref{eqn:ABC}. In particular, $A+B+C=0.$ Then it has the following arithmetic invariants:
	
	\begin{itemize}
		\item $c_4=16(A^2+AB+B^2)=16(A^2-CB)=16(B^2-CA)=16(C^2-AB)$,
		\item $c_6=-32(A-B)(B-C)(C-A)$,
		\item The discriminant of $E$ is $\Delta_{E}=16(ABC)^2$,
		\item The $j$-invariant of $E$ is $j_{E}=\frac{c_4^3}{\Delta_{E}}=2^8\frac{(A^2+AB+B^2)^3}{A^2B^2C^2}$.
		
	\end{itemize}
\end{lemma}
\begin{proof}
This follows immediately from the definitions of the arithmetic invariants of elliptic curves and the fact that $A+B+C=0$.
\end{proof}

\begin{defn}\label{NPMP}
	Let $E/K$ an elliptic curve of conductor $\mathcal{N}_E$ and $p$ be a rational prime. For a prime ideal $\frakq$ of $K$ denote
	by $\Delta_\frakq$ the discriminant of a local minimal model for $E$ at $\frakq$. Let
	
	\begin{equation}\label{condNM}
		\mathcal{M}_p:=\prod_{\substack{\mathfrak{q}\mid  \mid  \mathcal{N}_E \\ p\mid  v_{\mathfrak{q}}(\Delta_\mathfrak{q})}}\mathfrak{q}, \qquad \mathcal{N}_p:=\frac{\mathcal{N}_E}{\mathcal{M}_p}.
	\end{equation}
\end{defn}

\begin{prop} \label{prop:semi2}For the Frey curve $E$, the quantities $\mathcal N_{E}$ and $\mathcal N_p$ are as below:
	
	$$\mathcal N_{E}=\beta^{f_{\beta}}\prod\limits_{\mathfrak{P}\mid  2} \mathfrak{P}^{e_\mathfrak{P}} \prod_{\substack{ \fp \mid  ABC  \\ \mathfrak{p}\nmid 2, \mathfrak{p}\neq \beta}}\mathfrak{p} , \quad \mathcal N_p={\mathcal{D}}\beta^{f'_{\beta}}\prod\limits_{\mathfrak{P}\mid  2} \mathfrak{P}^{e'_{\mathfrak{P}}}, $$ where ${\mathcal{D}}:=\prod\limits_{\substack{ \fp \mid  (\mathcal{D}_0\mathcal{D}_1\mathcal{D}_2)   \\ \mathfrak{p}\nmid 2}}\mathfrak{p},$  
    and
    $0\leq e_{\mathfrak{P}}'\leq e_{\mathfrak{P}}\leq 2+6v_{\mathfrak{P}}(2)$ and $0\leq f_{\beta}'\leq f_{\beta} \leq 2$.
\end{prop}

\begin{proof}
For every $\frapk \nmid 2ABC$, the Frey curve $E$ has good reduction since $v_\frapk(\Delta_{E})=0$. For the primes $\frapk\nmid 2\beta$ and dividing $ABC$, $v_\frapk(\Delta_{E})>0$ and  $v_\frapk(c_4)=v_\frapk(A^2+AB+B^2)=0$ since $A,B,C$ are coprime away from $\beta$. Therefore the primes $\frapk\nmid 2\beta$ and $\frapk \mid  ABC$ have exponent one in the conductor of $E,$ hence we get the form for $\mathcal N_E$.% is semistable at these primes. %$\frapk\nmid 2\beta$ such that $\frapk\mid  ABC$.
  %This gives us the claimed form for $\mathcal N$. 
  
  To get $\mathcal N_p$ we need to show that for $\frapk \mid  ABC$ and $\frapk \nmid 2\beta$  we have $p\mid  v_\frapk(\Delta_{E})$. Note that for such $\frapk$, $v_\frapk(\Delta_{E})=2v_\frapk(ABC)=2pv_\frapk(\mathcal{I}_0\mathcal{I}_1\mathcal{I}_2)$. Hence the result follows.
\end{proof}

For the rest of this subsection, we assume that $r\mid  z$. In this case, we prove that the image of inertia at $\beta$ is large, or equivalently that $E$ has potentially multiplicative reduction at $\beta$.
\begin{proposition}\label{prop:valj}
	Let $D$ be relatively prime with $r$ and suppose that $v_r(z)=k$. Using the notation above, we get that $$ v_\beta(j_{E})=2-2(pk-1)(r-1).$$
In particular, by taking $p$ large enough, $E$ has potentially multiplicative reduction at $\beta$.	
	
\end{proposition}

\begin{proof} By Lemma \ref{lem:vals} we have $v_\beta(A)=(pk-1)(r-1)+1.$ Since $j_{E}=\frac{2^8(A^2+AB+B^2)^3}{A^2B^2C^2},$ and $v_\beta(B)=v_\beta(C)=2, v_\beta(A)>2$ by Lemma \ref{lem:vals}, we obtain $v_\beta(j_{E})=2-2(pk-1)(r-1).$

\end{proof}
Using Lemma 3.4 of \cite{FSasymp}, which we state for the convenience of
the reader below, we get the following result:
\begin{lemma}\cite[Lemma 3.4]{FSasymp}\label{inertia}
		Let $E$ be an elliptic curve over $K$ with $j$-invariant $j_E$. Let $p\geq 5$ and
		let $\mathfrak{q} \nmid p$ be a prime of $K$. Then $p \mid  \# \overline{\rho}_{E,p}(I_{\mathfrak{q}})$ if and only if $E$ has potentially
		multiplicative reduction at $\mathfrak{q}$ (i.e. $v_{\mathfrak{q}}(j_E)<0$) and $p \nmid v_{\mathfrak{q}}(j_E)$.
	\end{lemma}
	%\begin{proof}
	%	See \cite[Lemma 3.4]{FSasymp}.
	%\end{proof}

\begin{prop}\label{prop:inertia}
	Let $\beta$ as above (the prime lying over $r$) and $p$ big enough so that $\beta \nmid p$. Let $I_\beta$ denote the inertia group corresponding to the prime $\beta$. Then $p \mid  \#\overline{\rho}_{E,p}(I_\beta)$ provided that $\gcd(r,D)=1$.
	
\end{prop}

\begin{proof} By Proposition \ref{prop:valj} we have $v_\beta(j_{E})<0$ and  we see that $$p \nmid v_\beta(j_{E})=2-2(pk-1)(r-1)=-2pkr+2pk+2r.$$ Hence the result follows from Lemma \ref{inertia}.
	\end{proof}

\section{Theoretical Results}
In this section, we provide the theoretical results we need concerning modularity, irreducibility of mod $p$ Galois representations, level lowering and Conjectures \ref{ESconj} and \ref{weakFMconj} to prove our theorems.  We follow \cite{FSasymp} and the references therein.

Let $G_{K} $ be the absolute Galois group of a number field $K$, $E/K$ be an elliptic curve and $ \overline{\rho}_{E,p}$ denote the mod $p$ Galois representation of $E$.  We use $\frakq$ for an arbitrary prime of $K$, and $G_\frakq$ and $I_\frakq$ respectively for the decomposition and inertia subgroups of $G_K$ at $\frakq$.  Let $\frakf$ be a Hilbert eigenform over $K$.  We denote the field generated by its eigenvalues by $\Q_\frakf$ and a prime of $\Q_\frakf$ above $p$ by $\overline{\omega}$.

\subsection{Modularity}

Let $K$ be a totally real ﬁeld and $E$ be an elliptic curve over $K$, we say that $E$ is
modular if there exists a Hilbert cuspidal eigenform $\frakf$ over $K$ of parallel weight
2, with rational Hecke eigenvalues, such that the Hasse–Weil L-function of $E$ is
equal to the Hecke L-function of $\frakf$. In particular, this implies that the mod $p$
Galois representations are isomorphic, which we denote by $\overline{\rho}_{E,p} \sim \overline{\rho}_{\frakf,p}$.
We will use the following modularity theorem proved by Freitas, Hung, and
Siksek in \cite[Theorem 1]{FreitasLehungSiksek}.
\begin{theorem}\label{thm:modular}Let $K$ be a totally real field. There are at most finitely many $\overline{K}$
	isomorphism classes of non-modular elliptic curves $E$ over $K$. 
    %Moreover, if $K$ is real quadratic, then all elliptic curves over $K$ are modular.
\end{theorem}
%Furthermore, Derickx, Najman, and Siksek  proved the following result in \cite{DNS}:
%\begin{theorem} Let $K$ be a totally real cubic number field and $E$ be an elliptic curve over $K$. Then $E$ is modular.
%\end{theorem}

% Further results have been obtained for quartic totally real fields in
 %\cite{Box22}.
 %\begin{theorem}
 %	Let $E$ be an elliptic curve over a totally real
% 	quartic number field not containing a square root of $5$. Then E is modular.
% \end{theorem}
\begin{remark}
Conjecturally, all elliptic curves defined over totally real number fields are expected to be modular. Partial results towards this conjecture are known to hold true: for quadratic fields (\cite{FreitasLehungSiksek}), cubic fields (\cite{DNS}), quartic fields not containing a square root of $5$ (\cite{Box22}). 
\end{remark}

We note that the Frey elliptic curve $E$ defined in Section \ref{frey} depends on the prime $p$. A helpful trick that we will apply repeatedly is choosing 
$p$ sufficiently large, such that we can exclude finitely many problematic curves. For example, we can assume that 
$E$ is not one of the finitely many non-modular curves given by Theorem \ref{thm:modular}. We make this discussion precise in the following proposition.

\begin{proposition}\label{Problematic}
  Let $E$ be the elliptic curve defined in Section \ref{frey} and $S$ be a finite set of $j$-invariants. Then, for $p$ sufficiently large, we can assume that $j_{E}\notin S$.
\end{proposition}

\begin{proof}
      We denote the elements of the finite set $S$ as $j_1,\cdots, j_n$. Let $E$ be the Frey elliptic curve defined in Section \ref{frey} and recall the quantities $A=(\alpha_1-\alpha_2)f_0, B=(\alpha_2-2)f_1, C=(2-\alpha_1)f_2$.
     
     Note that the fact that $x,y,z$ are non-trivial implies that $A\neq 0$.  Let $\lambda=-B/A$, then the $j$-invariant of the Frey elliptic curve $E$ is $j(\lambda)=\frac{2^8(\lambda^2-\lambda+1)^3}{\lambda^2(\lambda-1)^2}$.  Suppose that $j(\lambda)=j_i\in S$, which can be viewed as a degree $6$ equation in $\lambda$. This equality holds true for finitely many values of $\lambda=-B/A$. 
     Rearranging the terms we get that
	$$\frac{f_1}{f_0}=\frac{x^2+y^2+\alpha_1xy}{(x+y)^2}= -\lambda \frac{\alpha_1-\alpha_2}{\alpha_2-2}.$$
The right-hand side of the equality is a constant multiple of $\lambda$ and hence, a member of a finite set. Now let $w=x/y$. This gives us the fact that $\frac{w^2+\alpha_1 w+1}{(w+1)^2}=-\lambda \frac{\alpha_1-\alpha_2}{\alpha_2-2}$ has finitely many solutions when $w\neq -1.$  Therefore $w=x/y$ has finitely many possibilities. Recall that $x$ and $y$ are coprime integers. Therefore, $x$ and $y$ are members of a finite set of integers. Since $r$ is a fixed rational prime, $x^r+y^r$ is a member of a finite set of integers too. Therefore, if  $(x,y,z)$ is a non-trivial, primitive integer solution of the equation $x^r+y^r=Dz^p$, it follows that $p$ has to be bounded. We quickly note that the assumption that $z\neq \pm 1$ was essential for this step. Thus by taking $p$ sufficiently large, we can assume that $j_E \notin S$.
  
\end{proof}

\begin{cor}\label{cor:modular}
    If $(x,y,z) \in \mathbb Z^3$ is a non-trivial primitive solution of $x^r+y^r=Dz^p$ where $p$ is sufficiently large, then the elliptic curve $E$ is modular. 
\end{cor}
\begin{proof}
    By Theorem \ref{thm:modular}, there are only finitely many non-modular $j$-invariants since $K$ is totally real. Now apply Proposition \ref{Problematic} with $S$ being the finite set of all non-modular $j$-invariants.
\end{proof}

\subsection{Irreducibility of mod $p$ representations of elliptic curves}
In this section, we state the result, which assures that the mod $p$ Galois representations associated with the Frey elliptic curve are irreducible. This is needed for the level lowering step. The below theorem
can be found in \cite{FSirred} as Theorem 2 deduced from the work of David, Momose and Merel’s
uniform boundedness theorem.
\begin{theorem}\label{thm:irr} Let $K$ be a Galois totally real field. There is an effective constant $C_K$,
	depending only on $K$, such that the following holds: If $p > C_K$ is a prime and $E/K$ is an
	elliptic curve which is semistable at all $\frakq\mid  p$, then $\overline{\rho}_{E,p}$is irreducible.
\end{theorem}
Note that the above theorem is also true for any totally real field by replacing $K$
by its Galois closure.

\begin{rmk}In our case, since $r$ is fixed the degree and class number of $K$ are fixed.  Moreover, by Proposition \ref{prop:semi2}, $E$ is semistable for $p > r$. Therefore, by the above theorem we get an explicit bound where $\bar{\rho}_{E,p}$ is irreducible for $p$ big enough.
\end{rmk}
\subsection{Level Lowering}
We present a level lowering result proved by Freitas and Siksek in \cite{FSasymp} derived
from the work of Fujira, Jarvis, and Rajaei. This should be viewed as a generalization of Ribet's Level Lowering Theorem and has the role of relating our Frey elliptic curve to a Hilbert newform with level which is independent of the putative solution.
Recall the definitions of $\mathcal M_p$ and  $\mathcal N_p$.  For an elliptic curve $E/K$ of conductor $\mathcal{N}_E$, define

		$$\mathcal{M}_p=\prod_{\substack{\mathfrak{q}\mid  \mid  \mathcal{N}_E \\ p\mid  v_{\mathfrak{q}}(\Delta_\mathfrak{q})}}\mathfrak{q}, \qquad \mathcal{N}_p=\frac{\mathcal{N}_E}{\mathcal{M}_p},$$
where $p$ be a rational prime and  $\Delta_\frakq$ is the discriminant of a local minimal model for $E$ at $\frakq$.

%We use the following generalization of Ribet's Level Lowering Theorem \cite{FSasymp}. % ,which can be found in \cite{FSasymp}.

\begin{theorem}\label{levelred}(Level Lowering)
	Let $K$ be a totally real field, and $E/K$ an elliptic
	curve of conductor $\mathcal N_E$ . For a prime ideal $\frakq$ of $K$ denote by $\Delta_\frakq$ the discriminant of a local minimal model for $E$ at $\frakq$. Let $p$ be a rational prime and $\mathcal M_p, \;\mathcal N_p$ defined as above.

	Suppose the following:
	\begin{enumerate}
		\item $p \geq  5$, the ramification index $e(\frakq/p) <p-1$ for all $\frakq\mid  p$, and $\mathbb Q(\zeta_p+\zeta_p^{-1}) \nsubseteq K,$ \
		
		\item $E$ is modular,
		\item $\bar{\rho}_{E,p}$ is irreducible,
		\item E is semistable at all $\frakq\mid  p$,
		\item  $p\mid  v_\frakq(\Delta_\frakq)$ for all $\frakq\mid  p$.
	\end{enumerate}
	
	Then, there is a Hilbert eigenform $\frakf$ of parallel weight $2$ that is new at level $\mathcal{N}_p$
	and some prime $\overline{\omega}$ of $\Q_\frakf$ such that $\overline{\omega }\mid  p$ and $\bar{\rho}_{E,p} \sim \bar{\rho}_{\frakf,\overline{\omega}}$.
\end{theorem}
\subsection {Conjectures}\label{sec: conj}
We will make use of the following famous conjecture which can be viewed as a converse of modularity. 
\begin{conj} \label{ESconj} (Eichler–Shimura). Let $K$ be a totally real field. Let $\frakf$  be a Hilbert
	newform of level $\mathcal N$ and parallel weight $2$, with rational eigenvalues. Then there is
	an elliptic curve $E_\frakf/K$ with conductor $\mathcal N$ having the same $L$-function as $\frakf$.
\end{conj}

It is known that this conjecture holds in some cases. Here we try to summarize some of the the results in this direction.

\begin{theorem} \label{thm:EC}
	Let $K$ be a totally real field and let $\frakf$  be a Hilbert
	newform over $K$ of level $\mathcal N$ and parallel weight $2$, such that $\Q_\frakf = \Q$. 
	
	\begin{enumerate}
		\item If the degree of $K$ over $\Q$ is odd then there is an elliptic curve $E_\frakf/K$ of conductor $\mathcal N$ with the same $L$-function as $\frakf$. (Blasius \cite{Blasius}, Hida \cite{Hida})
		
		\item Let $p$ be an odd prime and $\frakq\nmid p$ be a prime of $K$.  Suppose $\bar{\rho}_{E,p}$ is irreducible, and $\bar{\rho}_{E,p} \sim \bar{\rho}_{\frakf,p}$ with the following properties:
		\begin{enumerate}
			\item $E$ has potentially multiplicative reduction at $\frakq$,
			\item $p \mid  \# \bar{\rho}_{E,p}(I_\frakq) $ where $I_\frakq$ denotes the inertia group at $p$,
			\item $p \nmid (\text{Norm}_{K/\Q}(\frakq) \pm 1)$,
		\end{enumerate}
		then there is an elliptic curve $E_\frakf/K$ of conductor $\mathcal N$ with the same $L$-function as $\frakf$. (Freitas-Siksek Corollary 2.2 in \cite{FSasymp})
	\end{enumerate}
\end{theorem}
A famous conjecture in the field (e.g. \cite{MR1638494FreyMazur}) asserts that if $p$ is large enough, the modulo $p$ Galois representation attached to an elliptic curve determines it up to isogeny.

\begin{conj}\label{weakFMconj} (Weak Frey--Mazur). Let $K$ be a number field and $E_0$ be a fixed elliptic curve over $K$. Then there is a constant $N_{E_0}$, that depends only on $E_0$ such that if $p>N_{E_0}$ and $E/K$ is an elliptic curve with $\bar{\rho}_{E,p} \sim \bar{\rho}_{E_0,p}$ then $E$ is isogenous to $E_0$.
\end{conj}

We note that this is a known consequence of the \textit{abc}-conjecture. 
The strong version of the conjecture asserts that the constant $C:=N_{E_0}$ is independent of the starting curve $E_0$.
Over the rationals (i.e. $K=\Q$), Cremona and Freitas \cite[Theorem 1.3.]{cremona} have shown that for curves of conductor at most 500 000, the strong version of the Frey--Mazur Conjecture holds with $C=17$.

\section {Proofs of the Main theorems}

Firstly, we recall that $p$ is large enough such that we can assume that $D$ contains no $p$th power.
Let $(x,y,z)$ be a primitive non-trivial solution to $x^r+y^r=Dz^p.$ Moreover, whenever 
$r \equiv 1 \bmod 4$ we assume that $r\mid  z$.
%We have two cases:
% \begin{enumerate}
%    \item if $r\equiv 3 \mod 4$, then we define the Frey curve using $A,B,C$ as in Section \ref{curve2}. 
%    \item if $r \equiv 1 \mod 4$, then we moreover assume that $r\mid  z$ and again define the Frey curve using $A,B,C$ as in Section \ref{curve2}. 
%\end{enumerate}

\begin{remark}\label{rmk:r}
A crucial ingredient in the following proofs is applying the Eichler--Shimura correspondence. This step reduces the problem of finding solutions for Fermat-type equations, to a problem about elliptic curves with specified conductors.

As discussed in Section \ref{sec: conj}, the Eichler--Shimura conjecture is known to work only in specific cases. In the proof of the first Main Theorem, we distinguish two cases. Whenever $r\equiv 3 \mod 4$, one gets that $K:=\Q(\zeta_r +\zeta_r^{-1})$ has odd degree, and hence Theorem \ref{thm:EC} (1) tells us that the Eichler--Shimura conjecture holds unconditionally in this case. 
However, if $r \equiv 1 \mod 4$
, the situation changes. In this case, assuming $r\mid  z$
 and using an image of inertia comparison, we can show that 
$E$ satisfies conditions in Theorem \ref{thm:EC} (2), allowing us to apply the Eichler--Shimura Conjecture to conclude the proof.

\end{remark}

\subsection{Proof of First Main Theorem (Theorem \ref{thm:main1})}

Let $\frakq$ be a prime of $K$ lying over $p$. By Proposition   \ref{prop:semi2} we know 
that $E$ is semistable at all primes $\frakq\nmid 2, \beta$. Therefore by taking $p$ to be larger than an effective constant 
$\cB_{r,D}$, we can apply Theorem \ref{thm:irr} to get that 
$\bar{\rho}_{E,p}$ is irreducible. By applying Proposition \ref{prop:semi2}, Corollary \ref{cor:modular} and possibly enlarging $p$ once again, we see that the assumptions of Theorem  \ref{levelred} are satisfied.  Hence, we have that $\bar{\rho}_{E,p} \sim \bar{\rho}_{\frakf,\overline{\omega }}$ where $\frakf$ is a Hilbert eigenform of parallel weight $2$ that is new at level $\mathcal N_p$ where $\mathcal N_p$ is as given in the Proposition \ref{prop:semi2} and  some prime $\overline{\omega}$ of $\Q_\frakf$ such that $\overline{\omega }\mid  p$.  Next,
after possibly enlarging $\cB_{r,D}$, we reduce to the case when $\Q_\frakf=\Q$. This step uses standard ideas originally due to Mazur that can be found in \cite[Section 4]{BS}, and so we omit the details.

Now, we apply Theorem \ref{thm:EC} in the following manner. If $r \equiv 3 \mod 4$, then $K:=\Q(\zeta_r +\zeta_r^{-1})$ has odd degree
and we apply Theorem \ref{thm:EC} (1). If $r \equiv 1 \mod 4$ then we use the Theorem \ref{thm:EC} (2) with $\frakq$ in the statement of the theorem being $\beta$. Using the assumption $r\mid  z$, Propositions \ref{prop:valj} and \ref{prop:inertia} and taking $p$ to be larger than the $\text{Norm}_{K/\Q}(\beta) \pm 1$ we see that assumptions of Theorem \ref{thm:EC} (2) are satisfied. Either way, we get that $\bar{\rho}_{E,p} \sim \bar{\rho}_{E',p}$ for some finitely many elliptic curves $E'$ (the ones with conductor $\mathcal{N}_p$). %However if $r \equiv 1 \mod 4$, we need to use Proposition \ref{prop:inertia} and implicitly the assumption that $r\mid z$ to apply the second part(Theorem \ref{thm:EC} (2)) . Either way, we get that $\bar{\rho}_{E_i,p} \sim \bar{\rho}_{E',p}$ for some finitely many elliptic curves $E'$ (the ones with conductor $\mathcal{N}_p$). 
Using the Weak Frey--Mazur Conjecture we can conclude that $E$ is isogenous to one of finitely many elliptic curves $E'$, by possibly enlarging $p$ again. In each isogeny class there are finitely many members so we can say that $E$ is isomorphic to one of the finitely many $E'$s, whose finite set of $j$-invariants we denote by $S$. Now, by applying Proposition \ref{Problematic}, it follows that by possibly increasing $p$, we can assume that $j_{E}$ is not in this set, hence reaching a contradiction. 

\subsection{Proof of Second Main Theorem (Theorem \ref{thm:main2})}
Let $E$ be the Frey elliptic curve defined in Section \ref{curve2}.
The proof of this theorem is the same as above up to the point when we apply Theorem \ref{thm:EC}. In this case, we do not have the assumption of Theorem \ref{thm:EC} to get an Eichler--Shimura relation. Therefore we continue by assuming the Eichler--Shimura Conjecture (Conjecture \ref{ESconj}) and get that  $\bar{\rho}_{E,p} \sim \bar{\rho}_{E',p}$ for finitely many elliptic curves $E'$ (the ones with conductor $\mathcal{N}_p$). 

Then we continue as above and by using the Weak Frey--Mazur Conjecture we can conclude that $E$ is isogenous to one of the finitely many elliptic curves $E'$ and in each isogeny class there are finitely many members so that $E$ is isomorphic to one of the finitely many $E'$s whose finite set of $j$-invariants we denote by $S$. By Proposition \ref{Problematic}, for $p$ large enough, we can assume that $j_{E}$ is not in this set, hence reaching a contradiction once again.

\bibliographystyle{plain}
\bibliography{signature_(rrp).bib}
\end{document}